\documentclass{amsart}
\usepackage{amsfonts,amssymb,amsmath,amsthm,mathrsfs}
\usepackage{url}
\usepackage{enumerate}

\newtheorem{theorem}{Theorem}[section]
\newtheorem{lemma}[theorem]{Lemma}

\theoremstyle{definition}

\newtheorem{remark}[theorem]{Remark}
\newtheorem{example}[theorem]{Example}
\numberwithin{equation}{section}

\author[G. Hu]{Guoen Hu}
\address{Guoen Hu, Department of  Applied  Mathematics, Zhengzhou Information Science and Technology Institute,
Zhengzhou 450001,
P. R. China}
\email{guoenxx@163.com}
\thanks{The research    was supported by
the NNSF of
China under grant $\#$11371370.}

\begin{document}

\title[weighted vector-valued estimate]{Weighted vector-valued estimates for a  non-standard Calder\'on-Zygmund operator}

\begin{abstract}
In this paper, the author considers the weighted vector-valued estimates for the  operator defined by
$$T_Af(x)={\rm p.\,v.}\int_{\mathbb{R}^n}\frac{\Omega(x-y)}{|x-y|^{n+1}}\big(A(x)-A(y)-\nabla A(y)\big)f(y){\rm d}y$$
and its corresponding maximal operator $T_A^*$, where $\Omega$ is homogeneous of degree zero, has vanishing moment of order one, $A$ is a function in $\mathbb{R}^n$ such that $\nabla A\in {\rm BMO}(\mathbb{R}^n)$. By a pointwise estimate for $\|\{T_Af_k(x)\}\|_{l^q}$, the author obtains some quantitative weighted  vector-valued estimate for $T_A$ and $T^*_A$.

\end{abstract}
\maketitle
\section{Introduction}

In the remarkable work \cite{mu}, Muckenhoupt characterized the class of weights $w$ such that the Hardy-Littlewood maximal operator $M$ satisfies  the weighted $L^p$  $(p\in (1,\,\infty))$ estimate
\begin{eqnarray}\|Mf\|_{L^{p,\,\infty}(\mathbb{R}^n,\,w)}\lesssim \|f\|_{L^p(\mathbb{R}^n,\,w)}.\end{eqnarray}
The inequality (1.1) holds if and only if $w$ satisfies the $A_p(\mathbb{R}^n)$ condition, that is,
$$[w]_{A_p}:=\sup_{Q}\Big(\frac{1}{|Q|}\int_Qw(x){\rm d}x\Big)\Big(\frac{1}{|Q|}\int_{Q}w^{-\frac{1}{p-1}}(x){\rm d}x\Big)^{p-1}<\infty,$$
where the  supremum is taken over all cubes in $\mathbb{R}^n$, $[w]_{A_p}$ is called the $A_p$ constant of $w$. Also, Muckenhoupt proved that $M$ is bounded on $L^p(\mathbb{R}^n,\,w)$ if and only if $w$ satisfies the $A_p(\mathbb{R}^n)$ condition. Since then, considerable attention has been paid to the theory of $A_p(\mathbb{R}^n)$ and the weighted norm inequalities with $A_p(\mathbb{R}^n)$ weights for main operators in Harmonic Analysis, see \cite[Chapter 9]{gra} and related references therein.

However, the classical results on the weighted norm inequalities with $A_p(\mathbb{R}^n)$ weights did not reflect the quantitative dependence of the $L^p(\mathbb{R}^n,\,w)$ operator norm in terms of the relevant constant involving the weights. The question of the sharp dependence of the weighted estimates in terms of the $A_p(\mathbb{R}^n)$ constant specifically raised by Buckley \cite{bu}, who proved that if $p\in (1,\,\infty)$ and $w\in A_{p}(\mathbb{R}^n)$, then
\begin{eqnarray}\label{equa:1.2}\|Mf\|_{L^{p}(\mathbb{R}^n,\,w)}\lesssim_{n,\,p}[w]_{A_p}^{\frac{1}{p-1}}\|f\|_{L^{p}(\mathbb{R}^n,\,w)}.\end{eqnarray}
Moreover, the estimate (\ref{equa:1.2}) is sharp since the exponent $1/(p-1)$ can not be replaced by a smaller one. Hyt\"onen and P\'erez \cite{hp} improved the estimate (\ref{equa:1.2}), and showed that
\begin{eqnarray}\|Mf\|_{L^{p}(\mathbb{R}^n,\,w)}\lesssim_{n,\,p}\big([w]_{A_p}[w^{-\frac{1}{p-1}}]_{A_{\infty}}\big)^{\frac{1}{p}}\|f\|_{L^{p}(\mathbb{R}^n,\,w)}.\end{eqnarray}
where and in the following, for a weight $u$, $[u]_{A_{\infty}}$ is defined by
$$[u]_{A_{\infty}}=\sup_{Q\subset \mathbb{R}^n}\frac{1}{u(Q)}\int_{Q}M(u\chi_Q)(x){\rm d}x.$$
It is well known that for $w\in A_p(\mathbb{R}^n)$, $[w^{-\frac{1}{p-1}}]_{A_{\infty}}\lesssim [w]_{A_p}^{\frac{1}{p-1}}$. Thus, (1.3) is more subtle than (1.2).

The sharp dependence of the weighted estimates of  singular integral operators in terms of the $A_p(\mathbb{R}^n)$ constant  was much more complicated.  Petermichl \cite{pet1,pet2} solved this question for Hilbert transform and Riesz transform.   Hyt\"onen \cite{hyt}
proved that  for a  Calder\'on-Zygmund operator $T$ and $w\in A_2(\mathbb{R}^n)$,
\begin{eqnarray}\|Tf\|_{L^{2}(\mathbb{R}^n,\,w)}\lesssim_{n}[w]_{A_2}\|f\|_{L^{2}(\mathbb{R}^n,\,w)}.\end{eqnarray}
This solved  the so-called $A_2$ conjecture. Combining the estimate (1.4) and the extrapolation theorem in \cite{dra}, we know that
for a Calder\'on-Zygmund operator $T$, $p\in (1,\,\infty)$ and $w\in A_p(\mathbb{R}^n)$,
\begin{eqnarray}\|Tf\|_{L^{p}(\mathbb{R}^n,\,w)}\lesssim_{n,\,p}[w]_{A_p}^{\max\{1,\,\frac{1}{p-1}\}}\|f\|_{L^{p}(\mathbb{R}^n,\,w)}.\end{eqnarray}
In \cite{ler3}, Lerner  gave a much simplier proof of (1.4) by  controlling the Calder\'on-Zygmund operator using sparse operators.

Now let us consider a class of non-standard Calder\'on-Zygmund operators. For $x\in \mathbb{R}^n$, we denote by $x_j$ $(1\leq j\leq n)$ the $j$-th variable
of $x$. Let $\Omega$ be homogeneous of degree zero, integrable on the unit sphere $S^{n-1}$ and satisfy the
vanishing condition  that for all $1\leq j\leq n$,
\begin{eqnarray}\label{eq1.6}\int_{S^{n-1}}\Omega(x')x_j'{\rm d}x=0.\end{eqnarray} Let $A$ be a function on $\mathbb{R}^n$ whose  derivatives of order one in BMO$(\mathbb{R}^n)$. Define the  operator $T_A$ by
\begin{eqnarray}\label{eq:1.7} T_Af(x)={\rm p. v.}\int_{\mathbb{R}^n}\frac{\Omega(x-y)}{|x-y|^{n+1}}\big(A(x)-A(y)-\nabla A(y)(x-y)\big)f(y){\rm d}y.\end{eqnarray}
The maximal singular integral operator associated with $T_A$ is defined by
$$T_A^*f(x)=\sup_{\epsilon>0}\big|T_{A,\,\epsilon}f(x)|,$$
with $$T_{A,\,\epsilon}f(x)=\int_{|x-y|\geq \epsilon}
\frac{\Omega(x-y)}{|x-y|^{n+1}}\big(A(x)-A(y)-\nabla A(y)(x-y)\big)f(y){\rm d}y.$$
The operator $T_A$ is closed related to the Calder\'on commutator, of interest in PDE, and was first consider by  Cohen \cite{cohen}.  Cohen proved that if $\Omega\in{\rm Lip}_\alpha(S^{n-1})$ ($\alpha\in (0,\,1]$), then for $p\in (1,\,\infty)$, $T_A^*$ is a bounded operator
on $L^p(\mathbb{R}^n)$ with bound $C\|\nabla A\|_{{\rm BMO}(\mathbb{R}^n)}$. In fact, the argument in \cite{cohen} also leads to the boundedness on $L^p(\mathbb{R}^n,\,w)$ ($w\in A_p(\mathbb{R}^n))$ for $T_{A}$. Hofmann \cite{hof}
improved the result of Cohen and showed that $\Omega\in \cup_{q>1}L^q(S^{n-1})$ is a sufficient condition such that $T_A$ is bounded on $L^p(\mathbb{R}^n)$ for $p\in (1,\,\infty)$.  Hu and Yang \cite{hy} established the endpoint estimate for $T_A$,
from which they deduced some weighted $L^p$ estimates with general weights for $T_A$.

The purpose of this paper is to establish refined weighted vector-valued estimates for the operators $T_A$ and $T_A^*$. To formulate our result, we first recall  some definitions. Let $\Omega$ be a bounded function on $S^{n-1}$. The $L^{\infty}$ continuity modulus of $\Omega$ is defined by
$$\omega_{\infty}(t)=\sup_{|\rho|<t}|\Omega(\rho x')-\Omega(x')|,$$
where the supremum is taken over all rotations $\rho$ on the unit sphere $S^{n-1}$, and $|\rho|=\sup_{x'\in S^{n-1}}|\rho x'-x'|.$ Let $p,\,r\in(0,\,\infty]$ and $w$ be a weight. As usual, for a sequence of numbers $\{a_k\}_{k=1}^{\infty}$, we denote $\|\{a_k\}\|_{l^r}=\big(\sum_k|a_k|^r\big)^{1/r}$. The space $L^p(l^{r};\,\mathbb{R}^n,\,w)$ is defined as
$$L^p(l^{r};\,\mathbb{R}^n,\,w)=\big\{\{f_k\}_{k=1}^{\infty}:\, \|\{f_k\}\|_{L^p(l^r;\,\mathbb{R}^n,\,w)}<\infty\big\}$$
where
$$\|\{f_k\}\|_{L^p(l^r;\,\mathbb{R}^n,\,w)}=\Big(\int_{\mathbb{R}^n}\|\{f_k(x)\}\|_{l^r}^pw(x)\,{\rm d}x\Big)^{1/p}.$$
The space $L^{p,\,\infty}(l^{r};\,\mathbb{R}^n,\,w)$ is defined as
$$L^{p,\,\infty}(l^{r};\,\mathbb{R}^n,\,w)=\big\{\{f_k\}_{k=1}^{\infty}:\, \|\{f_k\}\|_{L^{p,\,\infty}(l^r;\,\mathbb{R}^n,\,w)}<\infty\big\}$$
with
$$\|\{f_k\}\|_{L^{p,\,\infty}(l^r;\,\mathbb{R}^n,\,w)}^p=\sup_{\lambda>0}\lambda^pw\Big(\Big\{x\in\mathbb{R}^n:\,
\|\{f_k(x)\}\|_{l^r}>\lambda\Big\}\Big).$$ When $w\equiv 1$, we denote  $\|\{f_k\}\|_{L^p(l^r;\,\mathbb{R}^n,\,w)}$ ($\|\{f_k\}\|_{L^{p,\infty}(l^r;\,\mathbb{R}^n,\,w)}$)
by $\|\{f_k\}\|_{L^p(l^r;\,\mathbb{R}^n)}$ ($\|\{f_k\}\|_{L^{p,\infty}(l^r;\,\mathbb{R}^n)}$) for simplicity.
Our first result can be stated as follows.
\begin{theorem}\label{t1.1} Let $\Omega$ be homogeneous of degree zero, satisfy the vanishing moment (1.6), $A$ be a function  in $\mathbb{R}^n$ whose  derivatives of order one in ${\rm BMO}(\mathbb{R}^n)$. Suppose that  the $L^{\infty}$ continuity modulus of $\Omega$ satisfies that
\begin{eqnarray}\label{eq1.8}\int^1_0\omega_{\infty}(t)(1+|\log t|)\frac{{\rm d}t}{t}<\infty,\end{eqnarray}
then for $p,\,q\in (1,\,\infty)$ and $w\in A_{p}(\mathbb{R}^n)$,
\begin{eqnarray*}
&&\big\|\{T_{A}f_k\}\big\|_{L^p(l^q;\mathbb{R}^n,w)}+\big\|\{T_{A}^*f_k\}\big\|_{L^p(l^q;\mathbb{R}^n,w)}\\
&&\quad\lesssim_{n,\,p}\|\nabla A\|_{{\rm BMO}(\mathbb{R}^n)} [w]_{A_p}^{\frac{1}{p}}\big([\sigma]_{A_{\infty}}^{\frac{1}{p}}+[w]_{A_{\infty}}^{\frac{1}{p'}}\big)[\sigma]_{A_{\infty}}
\|\{f_k\}\|_{L^p(l^q,\,\mathbb{R}^n,\,w)}.\end{eqnarray*}
with $\sigma=w^{-\frac{1}{p-1}}$. In particular,
\begin{eqnarray*}
&&\|\{T_{A}f_k\}\|_{L^p(l^q;\,\mathbb{R}^n,\,w)}+\|\{T_{A}^*f_k\}\|_{L^p(l^q;\,\mathbb{R}^n,\,w)}\\
&&\quad\lesssim_{n,\,p}\|\nabla A\|_{{\rm BMO}(\mathbb{R}^n)} [w]_{A_p}^{\max\{1,\,\frac{1}{p-1}\}}[\sigma]_{A_{\infty}}\|\{f_k\}\|_{L^p(l^q,\mathbb{R}^n,w)}.\end{eqnarray*}
\end{theorem}
\begin{remark} Theorem \ref{t1.1} implies that for $p\in (1,\,\infty)$ and $w\in A_{p}(\mathbb{R}^n)$,
\begin{eqnarray}\|T_{A}f\|_{L^p(\mathbb{R}^n,\,w)}\lesssim_{n,\,p,\,q}\|\nabla A\|_{{\rm BMO}(\mathbb{R}^n)} [w]_{A_p}^{\max\{1,\,\frac{1}{p-1}\}+\frac{1}{p-1}}\|f\|_{L^p(\mathbb{R}^n,w)}.\end{eqnarray}
For the case $p\in (1,\,2]$, this estimate is sharp in the sense that the exponent $\frac{2}{p-1}$ can not be replaced by a smaller one, see Example \ref{e4.1}. The quantitative bound in  (1.9) is new, although we do not know if it is sharp for $p\in (2,\,\infty)$.
\end{remark}
We are also interested in  the weighted endpoint bounds for $T_A$ and $T_A^*$. We have that
\begin{theorem}\label{t1.2} Let $\Omega$ be homogeneous of degree zero, satisfy the vanishing moment (1.6), $A$ be a function  in $\mathbb{R}^n$ whose  derivatives of order one in ${\rm BMO}(\mathbb{R}^n)$. Suppose that  $\Omega$ satisfies (\ref{eq1.8}),
then for $q\in (1,\,\infty)$ and $w\in A_{1}(\mathbb{R}^n)$,
\begin{eqnarray*}
&&w(\big\{x\in\mathbb{R}^n:\,\big\|\{T_{A}f_k\}\big\|_{l^q}>\lambda\big\}\big)+w(\big\{x\in\mathbb{R}^n:\,\big\|\{T_{A}^*f_k\}\big\|_{l^q}>\lambda\big\}\big)\\
&&\quad\lesssim_{n,\|\nabla A\|_{{\rm BMO}(\mathbb{R}^n)}} [w]_{A_1}\Psi_2([w]_{A_{\infty}})
\int_{\mathbb{R}^n}\frac{\|\{f_k\}\|_{l^q}}{\lambda}\log \Big({\rm e}
+\frac{\|\{f_k\}\|_{l^q}}{\lambda}\Big)w(x){\rm d}x,\end{eqnarray*}with $\Psi_2(t)=\log^2({\rm e}+t)$.
\end{theorem}

In what follows, $C$ always denotes a
positive constant that is independent of the main parameters
involved but whose value may differ from line to line. We use the
symbol $A\lesssim B$ to denote that there exists a positive constant
$C$ such that $A\le CB$.  Constant with subscript such as $C_1$,
does not change in different occurrences. For any set $E\subset\mathbb{R}^n$,
$\chi_E$ denotes its characteristic function.  For a cube
$Q\subset\mathbb{R}^n$ and $\lambda\in(0,\,\infty)$, we use $\ell(Q)$ (${\rm diam}Q$) to denote the side length (diamter) of $Q$, and
$\lambda Q$ to denote the cube with the same center as $Q$ and whose
side length is $\lambda$ times that of $Q$. For $x\in\mathbb{R}^n$ and $r>0$, $B(x,\,r)$ denotes the ball centered at $x$ and having radius $r$. For locally integrable function $f$ and a cube $Q\subset \mathbb{R}^n$, $\langle f\rangle_{Q}$ denotes the mean value of $f$ on $Q$, that is, $\langle f\rangle_{Q}=|Q|^{-1}\int_Qf(y){\rm d}y.$
\section{Dominated by sparse operator}
Recall that  the standard dyadic grid in $\mathbb{R}^n$ consists of all cubes of the form $$2^{-k}([0,\,1)^n+j),\,k\in  \mathbb{Z},\,\,j\in\mathbb{Z}^n.$$
Denote the standard grid by $\mathcal{D}$.

As usual, by a general dyadic grid $\mathscr{D}$,  we mean a collection of cube with the following properties: (i) for any cube $Q\in \mathscr{D}$, it side length $\ell(Q)$ is of the form $2^k$ for some $k\in \mathbb{Z}$; (ii) for any cubes $Q_1,\,Q_2\in \mathscr{D}$, $Q_1\cap Q_2\in\{Q_1,\,Q_2,\,\emptyset\}$; (iii) for each $k\in \mathbb{Z}$, the cubes of side length $2^k$ form a partition of $\mathbb{R}^n$.

Let $\mathscr{D}$ be a dyadic grid and
$M_{\mathscr{D}}$ be the maximal operator defined by
$$M_{\mathscr{D}}f(x)=\sup_{Q\ni x\atop{Q\in\mathscr{D}}} \langle  |f|\rangle_Q.$$
For $\delta>0$, let $M_{\mathscr{D},\,\delta}f(x)=\big\{M_{\mathscr{D}}(|f|^{\delta})(x)\big\}^{\frac{1}{\delta}}$ and $M_{\delta}f(x)=\big\{M(|f|^{\delta})(x)\big\}^{\frac{1}{\delta}}$. Associated with $\mathscr{D}$, define the sharp maximal function $M^{\sharp}_{\mathscr{D}}$ as
$$M^{\sharp}_{\mathscr{D}}f(x)=\sup_{Q\ni x\atop{Q\in\mathscr{D}}}\inf_{c\in\mathbb{C}}\frac{1}{|Q|}\int_{Q}|f(y)-c|{\rm d}y.$$
For $\delta\in (0,\,1)$, let $ M_{\mathscr{D},\,\delta}^{\sharp}f(x)=\big[M^{\sharp}_{\mathscr{D}}(|f|^{\delta})(x)\big]^{1/\delta}.$
Repeating the argument in \cite[p. 153]{ste2}, we can verify that, if $\Phi$ is a increasing function on $[0,\,\infty)$ which satisfies the doubling condition that
$$\Phi(2t)\leq C\Phi(t),\,t\in [0,\,\infty),$$ then
\begin{eqnarray}\label{eq2.1}&&\sup_{\lambda>0}\Phi(\lambda)|\{x\in\mathbb{R}^n:|h(x)|>\lambda\}|\lesssim
\sup_{\lambda>0}\Phi(\lambda)|\{x\in\mathbb{R}^n:M_{\mathscr{D},\delta}^{\sharp}h(x)>\lambda\}|,\end{eqnarray}
provided that $\sup_{\lambda>0}\Phi(\lambda)|\{x\in\mathbb{R}^n:\,M_{\mathscr{D},\,\delta}h(x)>\lambda\}|<\infty$, and
\begin{eqnarray}\label{eq2.2}&&\sup_{\lambda>0}\Phi(\lambda)|\{x\in\mathbb{R}^n:M_{\mathscr{D}}h(x)>\lambda\}|\lesssim
\sup_{\lambda>0}\Phi(\lambda)|\{x\in\mathbb{R}^n:M_{\mathscr{D}}^{\sharp}h(x)>\lambda\}|,\end{eqnarray}
provided that $\sup_{\lambda>0}\Phi(\lambda)|\{x\in\mathbb{R}^n:\,M_{\mathscr{D}}h(x)>\lambda\}|<\infty$,
see also \cite{perez1}.

Let $\eta\in (0,\,1)$ and $\mathcal{S}$ be a family of cubes. We say that $\mathcal{S}$ is $\eta$-sparse,  if for each fixed $Q\in \mathcal{S}$, there exists a measurable subset $E_Q\subset Q$, such that $|E_Q|\geq \eta|Q|$ and $\{E_{Q}\}$ are pairwise disjoint.
Associated with  the sparse family $\mathcal{S}$ and constants $\beta\in[0,\,\infty)$, we define the sparse operator $\mathcal{A}_{\mathcal{S},\,L(\log L)^\beta}$  by
$$\mathcal{A}_{\mathcal{S},\,L(\log L)^{\beta}}f(x)=\sum_{Q\in\mathcal{S}}\|f\|_{L(\log L)^{\beta},\,Q}\chi_{Q}(x),$$
here and in the following,
for $\beta\in [0,\,\infty)$,
$$\|f\|_{L(\log L)^{\beta},\,Q}=\inf\Big\{\lambda>0:\,\frac{1}{|Q|}\int_{Q}\frac{|f(y)|}{\lambda}\log^{\beta}\Big(1+\frac{|f(y)|}{\lambda}\Big){\rm d}y\leq 1\Big\}.$$
We denote $\mathcal{A}_{\mathcal{S},\,L(\log L)^{1}}$ by $\mathcal{A}_{\mathcal{S},\,L\log L}$ for simplicity.

\begin{lemma}\label{le2.1}Let $p\in (1,\,\infty)$, $w\in A_p(\mathbb{R}^n)$ and $\sigma=w^{-1/(p-1)}$. Let $\mathcal{S}$ be a sparse family. Then
\begin{eqnarray}\label{eq2.3}\quad\|\mathcal{A}_{\mathcal{S},L(\log L)^{\beta}}f\|_{L^p(\mathbb{R}^n,\,w)}\lesssim [w]_{A_p}^{\frac{1}{p}}\big([w]_{A_{\infty}}^{\frac{1}{p'}}+[\sigma]_{A_{\infty}}^{\frac{1}{p}}\big)[\sigma]_{A_{\infty}}^{\beta}\|f\|_{L^p(\mathbb{R}^n,\,w)}.\end{eqnarray}
\end{lemma}

For the proof of Lemma \ref{le2.1}, see \cite{chenhu}.

As in \cite{ler3}, for a sublinear operator $T$, we define the associated grand maximal  operator $\mathcal{M}_T$  by
$$\mathcal{M}_{T}f(x)=\sup_{Q\ni x}{\rm ess}\sup_{\xi\in Q}|T(f\chi_{\mathbb{R}^n\backslash 3Q})(\xi)|.$$
where the supremum is taken over all cubes $Q\subset \mathbb{R}^n$ containing $x$.

\begin{lemma}\label{le2.2} Let $q\in (1,\,\infty)$ and $Q_0\subset \mathbb{R}^n$. Let $T$ be a sublinear operator. Suppose that $T$ is bounded on $L^q(\mathbb{R}^n)$.
Then
for a. e. $x\in Q_0$,
$$\big\|\{T(f_k\chi_{3Q_0})(x)\}\big\|_{l^q}\le  C\|\{f_k(x)\}\|_{l^q}+\|\{\mathcal{M}_{T}(f_k\chi_{3Q_0}(x)\}\|_{l^q}.$$
\end{lemma}
\begin{proof}
We employ the argument in \cite{ler3}. Let $x\in {\rm int}Q_0$  be a point of approximation continuity of $\|\{T_A(f_k\chi_{3Q_0})\}\|_{l^q}$.
For  $r,\,\epsilon>0$, the set
$$E_r(x)=\{y\in B(x,\,r):\, \Big|\|\{T(f_k\chi_{3Q_0})(x)\}\|_{l^q}-\|\{T(f_k\chi_{3Q_0})(y)\}\|_{l^q}\Big|<\epsilon\}$$
satisfies that
$\lim_{r\rightarrow 0}\frac{|E_r(x)|}{|B(x,\,r)|}=1.$ Denote by $Q(x,\,r)$  the smallest cube centered at $x$ and containing $B(x,\,r)$.
Let $r>0$  small enough such that $Q(x,\,r)\subset Q_0$. Then for $y\in E_r(x)$,
\begin{eqnarray*}\|\{T(f_k\chi_{3Q_0})(x)\}\|_{l^q}&<&\|\{T(f_k\chi_{3Q_0})(y)\}\|_{l^q}+\epsilon\\
&\leq& \|\{T(f_k\chi_{3Q(x,\,r)})(y)\}\|_{l^q}+\big\|\{\mathcal{M}_{T}(f_k\chi_{3Q_0})(x)\}\big\|_{l^q}+\epsilon.\end{eqnarray*}
The boundedness on $L^{q}(\mathbb{R}^n)$ of $T$ tells us that
\begin{eqnarray*}\big\|\big\{T(f_k\chi_{3Q_0})(x)\big\}\big\|_{l^q}&\leq &\Big(\frac{1}{|E_{s}(x)|}\int_{E_s(x)}\|\{T(f_k\chi_{3Q(x,\,r)})(y)\}\|_{l^q}^{q}{\rm d}y\Big)^{\frac{1}{q}}\\
&&+\|\{\mathcal{M}_{T}(f_k\chi_{3Q_0})(x)\}\|_{l^q}+\epsilon\\
&\leq &C\Big(\frac{1}{|3Q(x,\,r)|}\int_{3Q(x,\,r)}\|\{f_k(z)\}\|_{l^q}^{q}dz\Big)^{\frac{1}{q}}\\
&&+\|\{\mathcal{M}_{T}(f_k\chi_{3Q_0})(x)\}\|_{l^q}+\epsilon.\end{eqnarray*}
Letting $r\rightarrow 0$ then  leads to the desired conclusion.
\end{proof}
We are now ready to give our main result in this section.
\begin{theorem}\label{th2.1}
Let $q\in (1,\,\infty)$, $\beta\in[0,\,\infty)$, $T$ be a sublinear operator and $\mathcal{M}_T$  the corresponding grand maximal operator. Suppose that $T$ is bounded on $L^q(\mathbb{R}^n)$, and for some constants $C_1>0$ and any $\lambda>0$, \begin{eqnarray}\label{equ:2.4}&&
\big|\big\{y\in \mathbb{R}^n:\|\{\mathcal{M}_{T}f_k(y)\}\|_{l^q}>\lambda\big\}\big|\\
&&\quad\le C_1 \int_{\mathbb{R}^n}\frac{\|\{f_k(y)\}\|_{l^q}}{\lambda}\log^{\beta} \Big(1+\frac{\|\{f_k(y)\}\|_{l^q}}{\lambda}\Big){\rm d}y.\nonumber\end{eqnarray}Then for
$N\in\mathbb{N}$ and bounded functions $f_1,\,\dots,\,f_N$ with compact supports, there exists a $\frac{1}{2}\frac{1}{3^n}$-sparse family $\mathcal{S}$ such that for ${\rm a.\,\,e.}\,\,x\in\mathbb{R}^n$,
\begin{eqnarray}\label{equ:2.5}
\|\{Tf_k(x)\}\|_{l^q}\lesssim \mathcal{A}_{\mathcal{S},\,L(\log L)^{\beta}}\big(\|\{f_k\}\|_{l^q}\big)(x).
\end{eqnarray}
\end{theorem}

\begin{proof} We employ the ideas in \cite{ler3}.  We claim that for each cube $Q_0\subset \mathbb{R}^n$,  there exist pairwise disjoint cubes $\{P_j\}\subset \mathcal{D}(Q_0)$, such that $\sum_{j}|P_j|\leq \frac{1}{2}|Q_0|$, and for a. e. $x\in Q_0$,
\begin{eqnarray}\label{equ:2.6}
\big\|\big\{T(f_k\chi_{3Q_0})(x)\big\}\big\|_{l^q}\chi_{Q_0}(x)&\le &C\big\|\|\{f_k\}\|_{l^q}\big\|_{L (\log L)^{\beta};\,3Q_0}\\
&&+\sum_{j}\|\{T(f_k\chi_{3P_j})(x)\}\|_{l^q}\chi_{P_j}(x).\nonumber
\end{eqnarray}
Let $C_2\in (1,\,\infty)$ be a  constant which will be chosen later.
It  follows from (\ref{equ:2.4}) that
\begin{eqnarray*}&&\big|\big\{x\in Q_0: \|\{\mathcal{M}_{T}(f_k\chi_{3Q_0})(x)\}\|_{l^q}>C_2\big\|\|\{f_k\}\|_{l^q}\big\|_{L (\log L)^{\beta};3Q_0}\big\}\big|\\
&&\quad\leq C_1\int_{3Q_0}\frac{\|\{f_k(y)\}\|_{l^q}}{C_2\big\|\|\{f_k\}\|_{l^q}\big\|_{L (\log L)^{\beta};3Q_0}}\log^{\beta} \big({\rm e}+\frac{\|\{f_k(y)\}\|_{l^q}}{C_2\big\|\|\{f_k\}\|_{l^q}\big\|_{L (\log L)^{\beta};3Q_0}}\big){\rm d}y\nonumber\\
&&\quad \leq \frac{C_1}{C_2}\int_{3Q_0}\frac{\|\{f_k(y)\}\|_{l^q}}{\big\|\|\{f_k\}\|_{l^q}\big\|_{L (\log L)^{\beta};\,3Q_0}}\log^{\beta}  \big({\rm e}+\frac{\|\{f_k(y)\}\|_{l^q}}{\big\|\|\{f_k\}\|_{l^q}\big\|_{L (\log L)^{\beta};\,3Q_0}}\big){\rm d}y\nonumber\\
&&\quad\le 3^n\frac{C_1}{C_2}|Q_0|.\nonumber
\end{eqnarray*}
Let \begin{eqnarray*}
E&=&\big\{y\in Q_0:\, \|\{f_k(x)\}\|_{l^q}>C_2\big\|\|\{f_k\}\|_{l^q}\big\|_{L (\log L)^{\beta};\,3Q_0}\big\}\\
&&\cup \big\{y\in Q_0:\, \|\{\mathcal{M}_{T}(f_k\chi_{3Q_0})(y)\}\|_{l^q}>C_2\big\|\|\{f_k\}\|_{l^q}\big\|_{L (\log L)^{\beta};\,3Q_0}\big\}.\end{eqnarray*}
Then $|E|\leq \frac{1}{2^{n+2}}|Q_0|$ if we choose $C_2$ large enough.

Now on cube $Q_0$, we apply the Calder\'on-Zygmund decomposition to $\chi_{E}$ at level $\frac{1}{2^{n+1}}$, we then obtain pairwise disjoint cubes $\{P_j\}\subset \mathcal{D}(Q_0)$, such that
$$\frac{1}{2^{n+1}}|P_j|\leq |P_j\cap E|\leq \frac{1}{2}|P_j|$$
and $|E\backslash\cup_jP_j|=0$.  Observe that $\sum_j|P_j|\leq \frac{1}{2}|Q_0|$ and $P_j\cap E^c\not =\emptyset.$ Therefore,
$${\rm ess}\sup_{\xi\in P_j}\|\{T(f_k\chi_{3Q_0\backslash 3P_j})(\xi)\}\|_{l^q}\leq C_2\big\|\|\{f_k\}\|_{l^q}\big\|_{L (\log L)^{\beta};\,3Q_0}.$$
By Lemma \ref{le2.2}, we also have that for a. e. $x\in Q_0\backslash \cup_jP_j$,
$$\|\{T(f_k\chi_{3Q_0})(x)\}\|_{l^q}\leq C_2\big\|\|\{f_k\}\|_{l^q}\big\|_{L (\log L)^{\beta};\,3Q_0}.$$
Observe that
\begin{eqnarray*}
&&\big\|\{T(f_k\chi_{3Q_0})(x)\}\big\|_{l^q}\chi_{Q_0}(x)\le \big\|\{T(f_k\chi_{3Q_0})(x)\}\big\|_{l^q}\chi_{Q_0\backslash \big(\cup_jP_j\big)}(x)\\
&&\quad+\sum_j\big\|\{T(f_k\chi_{3Q_0\backslash 3P_j})(x)\}\big\|_{l^q}\chi_{P_j}(x)
+\sum_j\big\|\{T(f_k\chi_{3P_j})(x)\}\big\|_{l^q}\chi_{P_j}(x)\\
&&\quad\leq 2C_2\big\|\|\{f_k\}\|_{l^q}\big\|_{L (\log L)^{\beta};\,3Q_0}+\sum_j\big\|\{T_A(f_k\chi_{3P_j})(x)\}\big\|_{l^q}\chi_{P_j}(x).
\end{eqnarray*}
The inequality (\ref{equ:2.6}) now follows directly.

We can now conclude the proof of Theorem \ref{th2.1}. As it was proved in \cite{ler3}, the last estimate shows that there exists a $\frac{1}{2}$-sparse family $\mathcal{F}\subset \mathcal{D}(Q_0)$, such that for a. e. $x\in Q_0$,
\begin{eqnarray*}
\big\|\big\{T(f_k\chi_{3Q_0})(x)\big\}\big\|_{l^q}\chi_{Q_0}(x)\lesssim \sum_{Q\in \mathcal{F}}\big\|\|\{f_k\}\|_{l^{q}}\big\|_{L(\log L)^{\beta},\,3Q}\chi_{Q}(x).
\end{eqnarray*}
Recalling that $\{f_k\}_{1\leq k\leq N}$ are functions in $L^1(\mathbb{R}^n)$ with compact supports, we can take now a partition of $\mathbb{R}^n$ by cubes $Q_j$ such that $\cup_{k=1}^N{\rm supp}\,f_k\subset 3Q_j$
for each $j$ and obtain a $\frac{1}{2}$-
sparse family $\mathcal{F}_j\subset\mathcal{D}(Q_j)$ such that  for a. e. $x\in Q_j$,
$$
\big\|\big\{T(f_k\chi_{3Q_j})(x)\big\}\big\|_{l^q}\chi_{Q_j}(x)\lesssim \sum_{Q\in \mathcal\mathcal{F}_j}\big\|\|\{f_i^k\}\|_{l^{q}}\big\|_{L(\log L)^{\beta},\,3Q}\chi_{Q}(x).
$$
Setting $\mathcal{S}=\{3Q: Q\in\cup_j\mathcal{F}_j\}$, we see that (\ref{equ:2.5}) holds for $\mathcal{S}$ and a. e. $x\in\mathbb{R}^n$. \end{proof}

\section{Proof of Theorem \ref{t1.1}}

To prove our theorem \ref{t1.1}, we will employ some  lemmas.
\begin{lemma}\label{le3.1}
Let $A$ be a function on $\mathbb{R}^n$ with derivatives of order one in $L^q(\mathbb{R}^n)$ for some $q\in (n,\,\infty]$. Then
$$|A(x)-A(y)| \lesssim|x-y|\Big(\frac{1}{|I_x^y|}\int_{I_x^y}|\nabla A(y)|^qdy\Big)^{\frac{1}{q}},$$
where
$I_x^y$ is the cube centered at $x$ and having side length $2|x-y|.$
\end{lemma}
For the proof of Lemma \ref{le3.1}, see \cite{cohen}.

For a fixed $\beta\in [0,\,\infty)$, let $M_{L(\log L)^{\beta}}$ be the maximal operator defined by $$M_{L(\log L)^{\beta}}f(x)=\sup_{Q\ni x}\|f\|_{L(\log L)^{\beta},\,Q},$$
where the supremum is take over all cubes containing $x$.
It is well known (see \cite{perez1}) that for any $\lambda>0$,
\begin{eqnarray}\label{equ:3.1}|\{x\in\mathbb{R}^n:\,M_{L(\log L)^{\beta}}f(x)>\lambda\}|\lesssim \int_{\mathbb{R}^n}\frac{|f(x)|}{\lambda}\log^{\beta} \Big(1+\frac{|f(x)|}{\lambda}\Big){\rm d}x.\end{eqnarray}

\begin{lemma}\label{le3.2}
Let $l\in \mathbb{N}$ and $q\in (1,\,\infty)$.  Then the maximal operator $M_{L(\log L)^l}$ satisfies that
\begin{eqnarray}\label{equ:3.2}&&\big|\big\{x\in\mathbb{R}^n:\big\|\big\{M_{L(\log L)^{l}}f_k(x)\big\}\big\|_{l^q}>\lambda\big\}\big|\nonumber\\
&&\quad\lesssim \int_{\mathbb{R}^n}
\frac{\|\{f_k(x)\}\|_{l^q}}{\lambda}\log^{l}\Big(1+\frac{\|\{f_k(x)\}\|_{l^q}}{\lambda}\Big){\rm d}x.
\end{eqnarray}
\end{lemma}
\begin{proof} We only consider the case $l=1$. The case $l\geq 2$ can be proved in the same way. As it was pointed out in \cite{cana} (see also \cite{perez}) that
\begin{eqnarray}\label{equ:3.3}M_{L \log L}f(x)\approx M^{2}f(x),\end{eqnarray} with $M^2$  the operator $M$ iterated twice. Thus, it suffices to show the operator $M^l$ satisfies (\ref{equ:3.2}). On the other hand, by the well known one-third trick  (see \cite[Lemma 2.5]{hlp}), we only need to prove that, for each dyadic grid $\mathscr{D}$, the inequality
\begin{eqnarray}\label{equ:3.4}&&\big|\big\{x\in\mathbb{R}^n:\big\|\big\{M_{\mathscr{D}}(M_{\mathscr{D}}f_k)(x)\big\}
\big\|_{l^q}>1\big\}\big|\\
&&\quad\lesssim \int_{\mathbb{R}^n}
\|\{f_k(x)\}\|_{l^q}\log\big(1+\|\{f_k(x)\}\|_{l^q}\big){\rm d}x.\nonumber
\end{eqnarray}holds when $\{f_k\}$ is finite. As in the proof of Lemma 8.1 in \cite{csmp}, we can very that
for each cube $Q\in\mathscr{D}$, $\delta\in (0,\frac{1}{q})$ and $\lambda\in (0,\,1)$,
\begin{eqnarray}\label{equ3.5}
&&\inf_{c\in\mathbb{C}}\Big(\frac{1}{|Q|}\int_Q\Big|\|\{M_{\mathscr{D}}f_k(y)\}\|_{l^q}-c\Big|^{\delta}{\rm d}y\Big)^{\frac{1}{\delta}}\\
&&\quad\lesssim\Big(\frac{1}{|Q|}\int_{Q}\|\{M_{\mathscr{D}}(f_k\chi_Q)(y)\})\|_{l^q}^{\delta}{\rm d}y\Big)^{\frac{1}{\delta}}\lesssim\langle\|\{f_k\chi_Q\}\|_{l^q}\rangle_{Q},\nonumber
\end{eqnarray}
where in the last inequality, we invoked the fact that $M_{\mathscr{D}}$ is bounded from $L^1(l^q,\,\mathbb{R}^n)$ to $L^{1,\,\infty}(l^q,\,\mathbb{R}^n)$. This, in turn, implies that
\begin{eqnarray}\label{eq4.7}M_{\mathscr{D},\,\delta}^{\sharp}\big(\|\{M_{\mathscr{D}}f_k\}\|_{l^q}\big)(x)\lesssim M_{\mathscr{D}}\big(\|\{f_k\}\|_{l^q}\big)(x).\end{eqnarray}
Again by the argument used in the proof of Lemma 8.1 in \cite{csmp}, we can verify that for each cube $Q\in \mathscr{D}$,
$$\inf_{c\in\mathbb{C}}\frac{1}{|Q|}\int_Q\Big|\|\{M_{\mathscr{D}}f_k(y)\}\|_{l^q}-c\Big|{\rm d}y\lesssim \frac{1}{|Q|}\int_{Q}\|\{M(f_k\chi_Q)\}\|_{l^q}{\rm d}y.
$$Therefore,
\begin{eqnarray}\label{eq4.8}M^{\sharp}_{\mathscr{D}}\big(\|\{M_{\mathscr{D}}f_k\}\|_{l^q}\big)(x)\lesssim \sup_{Q\ni x}\langle \|\{M_{\mathscr{D}}(f_k\chi_Q\})\|_{l^q}\rangle_Q.
\end{eqnarray}

Now we claim that for each cube $Q$,
\begin{eqnarray}\label{equ:3.8}\langle \|\{M(f_k\chi_Q\})\|_{l^q}\rangle_Q\lesssim \big\|\|\{f_k\}\|_{l^q}\big\|_{L\log L,\,Q}.
\end{eqnarray}
Let $\Phi(t)=t\log^{-1} ({\rm e}+t^{-1})$. If we can prove (\ref{equ:3.8}), it then follows from (\ref{eq2.1}), (\ref{eq4.7}),  (\ref{eq2.2}), (\ref{eq4.8}) and (\ref{equ:3.8}) that
\begin{eqnarray*}
&&\big|\big\{x\in\mathbb{R}^n:\big\|\big\{M_{\mathscr{D}}(M_{\mathscr{D}}f_k)(x)\big\}\big\|_{l^q}>1\big\}\big|\\
&&\quad\lesssim\sup_{\lambda>0}\Phi(\lambda)\big|\{x\in\mathbb{R}^n:
M_{\mathscr{D},\,\delta}^{\sharp}\big(\big\|\big\{M_{\mathscr{D}}(M_{\mathscr{D}}f_k)\big\}\big\|_{l^q}\big)(x)>\lambda\}\big|\\
&&\quad\lesssim\sup_{\lambda>0}\Phi(\lambda)\big|\{x\in\mathbb{R}^n:
M_{\mathscr{D}}\big(\|\{M_{\mathscr{D}}f_k\}\|_{l^q}\big)(x)>\lambda\}\big|\\
&&\quad\lesssim\sup_{\lambda>0}\Phi(\lambda)\big|\{x\in\mathbb{R}^n:M_{\mathscr{D}}^{\sharp}\big(\|\{M_{\mathscr{D}}f_k\}
\|_{l^q}\big)(x)>\lambda\}\big|\\
&&\quad\lesssim\sup_{\lambda>0}\Phi(\lambda)\big|\{x\in\mathbb{R}^n:\,M_{L\log L}\big(\|\{f_k\}
\|_{l^q}\big)(x)>\lambda\}\big|\\
&&\quad\lesssim\int_{\mathbb{R}^n}
\|\{f_k(x)\}\|_{l^q}\log\big(1+\|\{f_k(x)\}\|_{l^q}\big){\rm d}x,
\end{eqnarray*}
which gives (\ref{equ:3.4}).

We now prove (\ref{equ:3.8}). We may assume that $\big\|\|\{f_k\}\|_{l^q}\big\|_{L\log L,\,Q}=1$, which implies that
$$\int_{Q}\|\{f_k(y)\}\|_{l^q}\log (1+\|\{f_k(y)\}\|_{l^q}){\rm d}y\leq |Q|.$$
On the other hand, checking the proof of the Fefferman-Stein maximal inequality (see \cite{fes}), we see that for each $\lambda>0$,
\begin{eqnarray*}
\big|\big\{x\in\mathbb{R}^n:\,\|\{Mh_k(x)\}\|_{l^q}>\lambda\big\}\big|&\lesssim&\frac{1}{\lambda^2}\int_{\{\|\{h_k(y)\}\|_{l^q}\leq \lambda\}}
\|\{h_k(y)\}\|_{l^q}^2{\rm d}y\\
&&+\frac{1}{\lambda}\int_{\{\|\{h_k(y)\}\|_{l^q}>\lambda\}}\|\{h_k(y)\}\|_{l^q}{\rm d}y.
\end{eqnarray*}
We now obtain that
\begin{eqnarray*}
\int_{Q}\|\{M(f_k\chi_Q)(y)\}\|_{l^q}{\rm d}y&=&\int_{\{y\in Q:\,\|\{M(f_k\chi_Q)(y)\}\|_{l^q}\leq 1\}}\|\{M(f_k\chi_Q(y)\})\|_{l^q}{\rm d}y\\
&&+\int_{\{y\in Q:\,\|\{M(f_k\chi_Q)(y)\}\|_{l^q}>1\}}\|\{M(f_k\chi_Q)(y)\}\|_{l^q}{\rm d}y\\
&\lesssim&|Q|+\int_{1}^{\infty}\int_{\{x\in Q:\|\{f_k(x)\}\|_{l^q}\leq\lambda\}}\|\{f_k(x)\}\|_{l^q}^2{\rm d}x\frac{{\rm d}\lambda}{\lambda^2}\\
&&+\int_{1}^{\infty}\int_{\{x\in Q:\|\{f_k(x)\}\|_{l^q}>\lambda\}}\|\{f_k(x)\}\|_{l^q}{\rm d}x
\frac{1}{\lambda}{\rm d}\lambda\\
&\lesssim&|Q|.
\end{eqnarray*}
This establishes (\ref{equ:3.8}) and completes the proof of Lemma \ref{le3.2}.
\end{proof}

Let $\Omega$ be homogeneous of degree zero. For each $j$ with $1\leq j\leq n$, define the operator $T_j$ as
\begin{eqnarray}\label{equ:3.9}T_jf(x)={\rm p.\,v.}\int_{\mathbb{R}^n}\frac{\Omega(x-y)(x_j-y_j)}{|x-y|^{n+1}}f(y){\rm d}y.
\end{eqnarray}
\begin{lemma}\label{le3.3}Let $q\in (1,\,\infty)$. $T_A$ be the operator defined by (\ref{eq:1.7}). Under the hypothesis of Theorem \ref{t1.1}, for each $\lambda>0$,
\begin{eqnarray}\label{equ:3.10}&&|\{x\in\mathbb{R}^n:\|\{T_Af_k(x)\}\|_{l^q}>\lambda\big\}|\\
&&\quad\lesssim \int_{\mathbb{R}^n}\frac{ \|\{f_k(x)\}\|_{l^q}}{\lambda}\log\Big(1+\frac{ \|\{f_k(x)\}\|_{l^q}}{\lambda}\Big){\rm d}x.\nonumber\end{eqnarray}
\end{lemma}
\begin{proof} We will employ the argument from \cite{aj}. Applying the Calder\'on-Zygmund decomposition to  $\|\{f_k(x)\}\|_{l^q}$ at level $\lambda$, we obtain a sequence of cubes $\{Q_j\}_j$ with disjoint interiors, such that
$$\lambda< \langle\|\{f_k\}\|_{l^q}\rangle_{Q_j}\leq 2^n\lambda,$$and $\|\{f_k(x)\}\|_{l^q}\lesssim \lambda$ for a. e. $x\in\mathbb{R}^n\backslash \cup_jQ_j$.
Let
$$g_k(x)=f_k(x)\chi_{\mathbb{R}^n\backslash \cup_jQ_j}(x)+\sum_j\langle f_k\rangle_{Q_j}\chi_{Q_j}(x),$$
and $$b_k(x)=f_k(x)-g_k(x)=\sum_{j}\big(f_k(x)-\langle f_k\rangle_{Q_j}\big)\chi_{Q_j}(x):=\sum_jb_{k,\,j}(x).$$Let $E_{\lambda}=\cup_{n}4n Q_j$.
By the fact that $\|\{g_k\}\|_{L^{\infty}(l^q;\,\mathbb{R}^n)}\lesssim\lambda$ and the assumption (ii), we have that
\begin{eqnarray*}&&|\big\{x\in\mathbb{R}^n:\|\{T_Ag_k(x)\}\|_{l^q}>\lambda/2\}|\\
&&\quad\lesssim |E_{\lambda}|+|\big\{x\in\mathbb{R}^n\backslash E_{\lambda}:\|\{T_Ag_k(x)\}\|_{l^q}>\lambda/2\}|\\
&&\quad\lesssim \lambda^{-1}\|\{f_k\}\|_{L^1(l^q;\,\mathbb{R}^n)}.\end{eqnarray*}
Thus, the proof of (\ref{equ:3.10}) can be reduced to showing that
\begin{eqnarray}\label{equ:3.11}&&\big|\big\{x\in\mathbb{R}^n\backslash E_{\lambda}:\, \|\{T_Ab_k(x)\}\|_{l^q}>\lambda/2\}\big|\\
&&\quad\lesssim \int_{\mathbb{R}^n}\frac{\|\{f_k(x)\}\|_{l^q}}{\lambda}\log \Big({\rm e}+\frac{\|\{f_k(x)\}\|_{l^q}}{\lambda}\Big)
{\rm d}x.\nonumber\end{eqnarray}

We now prove (\ref{equ:3.11}). For each fixed $j$, let $$A_j(y)=A(y)-\langle \nabla A\rangle_{Q_j}y.$$
We can write
\begin{eqnarray*}
T_Ab_{k}(x)&=&\sum_j\int_{\mathbb{R}^n}\frac{\Omega(x-y)}{|x-y|^{n+1}}\big(A_j(x)-A_j(y)\big)b_{k,\,j}(y){\rm d}y\\
&&+\sum_{i=1}^n\int_{\mathbb{R}^n}\Omega(x-y)\frac{x_i-y_i}{|x-y|^{n+1}}\sum_{j}\big(\partial_iA(y)-\langle\partial_iA\rangle_{Q_j}\big)b_{k,\,j}(y){\rm d}y\\
&:=&\sum_{j}T_A^1b_{k,\,j}(y)+\sum_{i=1}^nT_{A,\,i}b_k(y).
\end{eqnarray*}
Invoking Minkowski's inequality, we see that for each $j$,
$$\|\{b_{k,\,j}(x)\}\|_{l^q}\le \big(\|\{b_k\}\|_{l^q}+\lambda\big)\chi_{Q_j}.$$
By the vector-valued Calder\'on-Zygmund theory (see \cite{aj}), we see that for each fixed $1\leq i\leq n$,
\begin{eqnarray}\label{equa:3.12}
&&\big|\big\{x\in\mathbb{R}^n:\, \|\{T_{A,\,i}b_k(y)\}\|_{l^q}>\frac{\lambda}{4n}\big\}\big|\\
&&\quad\lesssim \lambda^{-1}\sum_j\int_{Q}|\nabla A(y)-\langle\partial_kA\rangle_{Q_j}|\|\{b_{k,\,j}(y)\}\|_{l^q}{\rm d}y\nonumber\\
&&\quad\lesssim\lambda^{-1}\sum_{j}|Q_j|\|\nabla A(y)-\langle\partial_kA\rangle_{Q_j}\big\|_{{\rm exp}L,\,Q_j}\big\|\|\{b_{k,\,j}\}\|_{l^q}\big\|_{L\log L,\,Q_j}\nonumber\\
&&\quad\lesssim\int_{\mathbb{R}^n}\frac{\|\{f_k(x)\}\|_{l^q}}{\lambda}\log\Big(1+\frac{\|\{f_k(x)\}\|_{l^q}}{\lambda}
\Big){\rm d}x.\nonumber
\end{eqnarray}
where
$$\|h\|_{{\rm exp}L,\,Q_j}=\inf\Big\{t>0:\,\frac{1}{|Q_j|}\int_{Q_j}{\rm exp}\Big(\frac{|h(y)|}{t}\Big){\rm d}y\leq 2\Big\},$$
and the second inequality follows from the generalization of H\"older's inequality (see \cite[p.64]{rr}), and the last inequality follows from the fact that $$\|h\|_{L\log L,\,Q_j}\le  \lambda+\frac{\lambda}{|Q_j|}\int_{Q_j}\frac{|h(y)|}{\lambda}\log\Big(1+\frac{|h(y)|}{\lambda}\Big){\rm d}y,$$
see \cite[p. 69]{rr}.

It remains to estimate $T_A^1b_{k}$. For each fixed $Q_j$, we choose $x_j\in 3Q_j\backslash 2Q_j$. By vanishing moment of $b_{j,\,k}$, we can write
\begin{eqnarray*}
|T_A^1b_{k,\,j}(x)|&\leq& \frac{1}{|x-x_j|^{n+1}}\int_{\mathbb{R}^n}|A_j(x_j)-A_j(y)||b_{j,\,k}(y)|{\rm d}y\\
&&+\int_{\mathbb{R}^n}\Big|\frac{\Omega(x-y)}{|x-y|^{n+1}}-\frac{\Omega(x-x_j)}{x-x_j|^{n+1}}\Big||A_j(x)-A_j(y)||b_{j,\,k}(y)|{\rm d}y\\
&:=&T_A^{{\rm I}}b_{k,\,j}(x)+T_A^{{\rm II}}b_{k,\,j}(x).
\end{eqnarray*}
By Lemma \ref{le3.1}, we have that
$$\sum_j|T_A^{{\rm I}}b_{k,\,j}(x)|\lesssim \sum_j\frac{1}{|x-x_j|^{n+1}}|Q_j|^{\frac{1}{n}}\|b_{k,\,j}\|_{L^1(\mathbb{R}^n)},$$
which via Minkowski's inequality implies that,
\begin{eqnarray*}
\Big(\sum_k\Big(\sum_j|T_A^{{\rm I}}b_{k,\,j}(x)|\Big)^q\Big)^{\frac{1}{q}}&\lesssim& \sum_j\frac{1}{|x-x_j|^{n+1}}|Q_j|^{\frac{1}{n}}\Big(\sum_k\|b_{k,\,j}\|_{L^1(\mathbb{R}^n)}^q\Big)^{\frac{1}{q}}\\
&\lesssim&\sum_j\frac{1}{|x-x_j|^{n+1}}|Q_j|^{\frac{1}{n}}\|\{b_{k,\,j}\}\|_{L^1(l^q;\,\mathbb{R}^n)}.
\end{eqnarray*}
Therefore,
\begin{eqnarray}\label{equa:3.13}&&\int_{\mathbb{R}^n\backslash E_{\lambda}}\Big(\sum_k\Big(\sum_j|T_A^{{\rm I}}b_{k,\,j}(x)|\Big)^q\Big)^{\frac{1}{q}}{\rm d}x\\
&&\quad\lesssim\sum_j\int_{\mathbb{R}^n\backslash 4nQ_j}\frac{|Q_j|^{\frac{1}{n}}}{|x-x_j|^{n+1}}{\rm d}x\|\{b_{k,\,j}\}\|_{L^1(l^q;\,\mathbb{R}^n)}\nonumber\\
&&\quad\lesssim\|\{b_{k}\}\|_{L^1(l^q;\,\mathbb{R}^n)}.\nonumber
\end{eqnarray}

To estimate $|T_A^{{\rm II}}b_{k,\,j}(x)|$, we first observe that if $y\in Q_j$ and $x\in 2^{d+1}nQ_j\backslash 2^dnQ_j$ with $d\geq 2$, then by Lemma \ref{le3.1},
\begin{eqnarray*}
|A_j(x)-A_j(y)|&\lesssim &|x-y|\Big(\frac{1}{|I_x^y|}\int_{I_x^y}|\nabla A(y)-\langle \nabla A\rangle_{Q_j}|^q{\rm d}y\Big)^{\frac{1}{q}}\\
&\lesssim&|x-y|\Big(\frac{1}{|I_x^y|}\int_{I_x^y}|\nabla A(z)-\langle \nabla A\rangle_{I_x^y}|^q{\rm d}z\Big)^{\frac{1}{q}}\\
&&+|x-y|\big|\langle \nabla A\rangle_{Q_j}-\langle \nabla A\rangle_{I_x^y}\big|\\
&\lesssim&d|x-y|.
\end{eqnarray*}
This, via the continuity condition (1.8), implies that for each $y\in Q_j$,
$$
\sum_{d=2}^{\infty}d\int_{2^{d+1}Q_j\backslash 2^dQ_j}
\Big|\frac{\Omega(x-y)}{|x-y|^{n+1}}-\frac{\Omega(x-x_j)}{x-x_j|^{n+1}}\Big||A_j(x)-A_j(y)|{\rm d}x\lesssim 1.
$$
On the other hand, another application of Minkowski's inequality gives us that
\begin{eqnarray*}
&&\Big(\sum_{k}\Big(\sum_j|T_A^{{\rm II}}b_{k,\,j}(x)|\Big)^q\Big)^{\frac{1}{q}}\lesssim\sum_{j}\Big(\sum_k|T_{A}^{{\rm II}}b_{k,\,j}(x)|^q\Big)^{\frac{1}{q}}\\
&&\quad\lesssim \sum_j\int_{\mathbb{R}^n}\Big|\frac{\Omega(x-y)}{|x-y|^{n+1}}-\frac{\Omega(x-x_j)}{x-x_j|^{n+1}}\Big||A_j(x)-A_j(y)|\|\{b_{k,\,j}(y)\}\|_{l^q}{\rm d}y.
\end{eqnarray*}
We thus deduce that\begin{eqnarray}\label{equa:3.14}&&\int_{\mathbb{R}^n\backslash E_{\lambda}}\Big(\sum_k\Big(\sum_j|T_A^{{\rm II}}b_{k,\,j}(x)|\Big)^q\Big)^{\frac{1}{q}}{\rm d}x\\
&&\quad\lesssim\sum_j\|\{b_{k,\,j}\}\|_{L^1(l^q;\,\mathbb{R}^n)}\lesssim\|\{b_{k}\}\|_{L^1(l^q;\,\mathbb{R}^n)}.\nonumber\end{eqnarray}
Combining the estimates (\ref{equa:3.13}) and (\ref{equa:3.14}) yields
$$\int_{\mathbb{R}^n\backslash E_{\lambda}}\Big(\sum_k\Big(\sum_j|T_A^{1}b_{k,\,j}(x)|\Big)^q\Big)^{\frac{1}{q}}{\rm d}x\lesssim\|\{b_{k}\}\|_{L^1(l^q;\,\mathbb{R}^n)},$$
which, via  the estimates  (\ref{equa:3.12}), leads to (\ref{equ:3.11}) and then completes the proof of Lemma \ref{le3.3}.
\end{proof}

Now let $\gamma\in (0,\,1]$. We know from Theorem 1 in \cite{hl} that,
$$T_A^*f(x)\lesssim_{\gamma} M_{\gamma}(T_Af)(x)+M_{L\log L}f(x).$$
For fixed $0<\delta<\gamma<1$, dyadic grid $\mathscr{D}$ and cube $Q\in \mathscr{D}$. Again as in \cite{csmp},
\begin{eqnarray*}
\inf_{c\in\mathbb{C}}\Big(\frac{1}{|Q|}\int_{Q}\Big|\|\{M_{\mathscr{D},\,\gamma}f_k(y)\}\|_{l^q}-c\Big|^{\delta}{\rm d}y
\Big)^{\frac{1}{\delta}}&\lesssim & \Big(\frac{1}{|Q|}\int_{Q}\|\{M_{\mathscr{D},\,\gamma}(f_k\chi_Q)\}\|_{l^q}^{\delta}\Big)^{\frac{1}{\delta}}\\
&\lesssim & \Big(\frac{1}{|Q|}\int_{Q}\|\{M(|f_k|^{\gamma}\chi_Q)\}\|_{l^{\frac{q}{\gamma}}}^{\frac{\delta}{\gamma}}\Big)^{\frac{1}{\delta}}\\
&\lesssim & \Big(\frac{1}{|Q|}\int_{Q}\|\{|f_k(y)|^{\gamma}\}\|_{l^{\frac{q}{\gamma}}}{\rm d}y\Big)^{\frac{1}{\gamma}},\\
\end{eqnarray*}
since $M$ is bounded from $L^1(l^{\frac{q}{\gamma}};\,\mathbb{R}^n)$ to $L^{1,\,\infty}(l^{\frac{q}{\gamma}};\,\mathbb{R}^n)$. Recall that by \ref{le3.3}, $T_A$ is bounded from $L^1(l^q;\,\mathbb{R}^n)$ to $L^{1,\,\infty}(l^q;\,\mathbb{R}^n)$. By (\ref{eq2.1}) and the argument used in the proof of Lemma \ref{le3.2}, we get that
\begin{eqnarray*}
&&\big|\big\{x\in\mathbb{R}^n:\big\|\big\{M_{\mathscr{D},\,\gamma}(T_Af_k)(x)\big\}\big\|_{l^q}>1\big\}\big|\\
&&\quad\lesssim\sup_{\lambda>0}\Phi(\lambda)\big|\{x\in\mathbb{R}^n:
M_{\mathscr{D},\,\delta}^{\sharp}\big(\big\|\big\{M_{\mathscr{D},\,\gamma}(T_Af_k)\big\}\big\|_{l^q}\big)(x)>\lambda\}\big|\\
&&\quad\lesssim\sup_{\lambda>0}\Phi(\lambda)\big|\{x\in\mathbb{R}^n:
M_{\gamma}\big(\|\{T_Af_k\}\|_{l^q}\big)(x)>\lambda\}\big|\\
&&\quad\lesssim\sup_{\lambda>0}\Phi(\lambda)\lambda^{-1}\sup_{s\geq 2^{-\frac{1}{\gamma}}\lambda}s\big|\{x\in\mathbb{R}^n:
\|\{T_Af_k(x)\}\|_{l^q}>s\}\big|\\
&&\quad\lesssim\int_{\mathbb{R}^n}
\|\{f_k(x)\}\|_{l^q}\log\big(1+\|\{f_k(x)\}\|_{l^q}\big){\rm d}x,
\end{eqnarray*}
where the second-to-last inequality follows from the  inequality (11) in \cite{hl}, and the last inequality follows from Lemma \ref{le3.3}. This, together with the one-third trick and Lemma \ref{le3.2}, leads to that
\begin{eqnarray}\label{equa:3.15}&&\big|\big\{x\in\mathbb{R}^n:\,\|\{T_A^*f_k(x)\}\|_{l^q}>\lambda\big\}\big|\\&&\quad
\lesssim\int_{\mathbb{R}^n}\frac{\|\{f_k(y)\}\|_{l^q}}{\lambda}\log \Big(1+\frac{\|\{f_k(y)\}\|_{l^q}}{\lambda}\Big){\rm d}y.\nonumber\end{eqnarray}

{\it Proof of Theorem \ref{t1.1}}. Let $q\in (1,\,\infty)$. Recall that $T_A$ is bounded on $L^q(\mathbb{R}^n)$. If we can prove that for all $x\in \mathbb{R}^n$,
\begin{eqnarray}\label{equa:3.16}
\mathcal{M}_{T_A}f(x)\le CM_{L\log L}f(x)+T^*_Af(x),\end{eqnarray}
then by Lemma \ref{le3.2} and (\ref{equa:3.15}),
$$\big|\big\{x\in\mathbb{R}^n:\,\|\{\mathcal{M}_{T_A}f_k(x)\}\|_{l^q}>2\lambda\big\}\big|
\lesssim\int_{\mathbb{R}^n}\frac{\|\{f_k(y)\}\|_{l^q}}{\lambda}\log \Big(1+\frac{\|\{f_k(y)\}\|_{l^q}}{\lambda}\Big){\rm d}y.$$
This, via Theorem \ref{th2.1}, implies that for bounded functions $f_1,\dots,f_N$ with compact supports, there exists a sparse family $\mathcal{S}$, such that for a.\,\,e. $x\in\mathbb{R}^n$,
\begin{eqnarray}\label{equa:3.15'}\big\|\{T_Af_k(x)\}\big\|_{l^q}\lesssim \mathcal{A}_{\mathcal{S},\,L\log L}(\|\{f_k\}\|_{l^q})(x).\end{eqnarray}
Our desired conclusion about $ T_A$ then follows from Lemma \ref{le2.1} directly.

We now prove (\ref{equa:3.16}). Let $Q\subset \mathbb{R}^n$ be a cube and $x,\,\xi\in Q$. Denote by $B_x$ the ball centered at $x$ and having diameter $20{\rm diam}\,Q$. As in \cite{ler3}, we can write
\begin{eqnarray*}
|T_A(f\chi_{\mathbb{R}^n\backslash 3Q})(\xi)|&\leq &|T_A(f\chi_{\mathbb{R}^n\backslash B_x})(z)-T_A(f\chi_{\mathbb{R}^n\backslash B_x})(\xi)|\\
&&+|T_A(f\chi_{\mathbb{R}^n\backslash B_x})(z)|+|T_A(f\chi_{B_x\backslash 3Q})(\xi)|.
\end{eqnarray*}
It is obvious that
\begin{eqnarray}\label{equa:3.17}|T_A(f\chi_{\mathbb{R}^n\backslash B_x})(x)|\leq T_A^*f(x).\end{eqnarray}
Let $A_{B_x}(y)=A(y)-\langle\nabla A\rangle _{B_x}y$.  We have that $$A(\xi)-A(y)-\nabla A(y)(\xi-y)=A_{B_x}(\xi)-A_{B_x}(y)-\nabla A_{B_x}(y)(\xi-y).$$
A trivial computation then leads to that
\begin{eqnarray*}
|T_A(f\chi_{B_x\backslash 3Q})(\xi)|&\lesssim&\int_{B_x\backslash 3Q}\frac{|A(\xi)-A(y)-\nabla A(y)(\xi-y)|}{|\xi-y|^{n+1}}|f(y)|{\rm d}y\\
&\lesssim&\frac{1}{|B_x|^{1+\frac{1}{n}}}\int_{B_x\backslash 3Q}\big|A_{B_x}(\xi)-A_{B_x}(y)\big||f(y)|{\rm d}y\\
&&+\frac{1}{|B_x|}\int_{B_x}\big|\nabla A(y)-m_{B_x}(\nabla A)\big||f(y)|{\rm d}y\\
&=&{\rm I}(\xi)+{\rm II}(\xi).
\end{eqnarray*}
Note that for $y\in B_x\backslash 3Q$ and $\xi\in Q$,
$I_{\xi}^y\subset B_x\subset 4nI_{\xi}^y.$
An application of Lemma \ref{le3.1} shows that $$\big|A_{B_x}(\xi)-A_{B_x}(y)\big|\lesssim |B_x|^{\frac{1}{n}},$$ and so
$${\rm I}(\xi)\leq \frac{1}{|B_x|}\int_{B_x}|f(y)|{\rm d}y\lesssim Mf(x).$$
On the other hand, by the generalization of H\"older's inequality and the John-Nirenberg inequality, we deduce that
$${\rm II}(\xi)\lesssim \|f\|_{L\log L,\,B_x}\lesssim M_{L\log L}f(x).$$
Therefore, \begin{eqnarray}\label{equa:3.18}|T_A(f\chi_{B_x\backslash 3Q})(\xi)|\lesssim M_{L\log L}f(x).\end{eqnarray}

To  estimate $|T_A(f\chi_{\mathbb{R}^n\backslash B_x})(x)-T_A(f\chi_{\mathbb{R}^n\backslash B_x})(\xi)|$, we employ the ideas used in \cite{cohen,hy}.
Write
\begin{eqnarray*}
&&\Big|\frac{\Omega(x-y)}{|x-y|^{n+1}}(A(x)-A(y)-\nabla A(y)(x-y))\\
&&\qquad\qquad-\frac{\Omega(\xi-y)}{|\xi-y|^{n+1}}(A(\xi)-A(y)-\nabla A(y)(\xi-y))\Big|\\
&&\quad\lesssim\Big|\frac{\Omega(x-y)}{|x-y|^{n+1}}-\frac{\Omega(\xi-y)}{|\xi-y|^{n+1}}\Big|
\big|A_{B_x}(\xi)-A_{B_x}(y)-\nabla A_{B_x}(y)(\xi-y)\big|\\
&&\qquad+\frac{|\Omega(x-y)|}{|x-y|^{n+1}}\big|A_{B_x}(x)-A_{B_x}(\xi)-\nabla A_{B_x}(y)(x-\xi)\big|\\
&&:=G(x,\,\xi)+H(x,\,\xi).
\end{eqnarray*}
Another application of Lemma \ref{le3.1} gives us that for $q\in (n,\,\infty)$,
\begin{eqnarray*}
\big|A_{B_x}(x)-A_{B_x}(\xi)\big|&\lesssim &|x-\xi|\Big(\frac{1}{|I_x^{\xi}|}\int_{I_{x}^{\xi}}\big|\nabla A(z)-\langle \nabla A\rangle_{B_x}\big|^qdz\Big)^{1/q}\\
&\lesssim&|x-\xi|\big(1+\big|\langle \nabla A\rangle_{B_x}-\langle \nabla A\rangle_{I_x^{\xi}}\big|\big)\\
&\lesssim&|x-\xi|\Big(1+\log\frac{\ell(Q)}{|x-\xi|}\Big).
\end{eqnarray*}
A trivial computation leads to that if $\xi\in Q\backslash \{x\}$, then
\begin{eqnarray*}
\int_{\mathbb{R}^n\backslash B_x}H(x,\xi)|f(y)|{\rm d}y&\lesssim & |x-\xi|\Big(1+\log\big(\frac{\ell(Q)}{|x-\xi|}\big)\Big)\int_{\mathbb{R}^n\backslash B_x}\frac{|f(y)|}{|x-y|^{n+1}}{\rm d}y\\
&&+|x-\xi|\int_{\mathbb{R}^n\backslash B_x}\frac{|\nabla A(y)-\langle \nabla A\rangle_{B_x}|}{|x-y|^{n+1}}|f(y)|{\rm d}y\\
&\lesssim&\frac{|x-\xi|}{\ell(Q)}\Big(1+\log\big(\frac{\ell(Q)}{|x-\xi|}\big)\Big)Mf(x)\\
&&++M_{L\log L}f(x)\lesssim M_{L\log L}f(x).
\end{eqnarray*}
For each $y\in 2^{k}B_x\backslash 2^{k-1}B_x$ with $k\in\mathbb{Z}$, we have by Lemma 4.1 that
$$\big|A_{B_x}(\xi)-A_{B_x}(y)-\nabla A_{B_x}(y)(\xi-y)\big|\lesssim\big(k+|\nabla A(y)-\langle \nabla A\rangle_{B_x}|\big).$$
This, in turn leads to that
\begin{eqnarray*}
\int_{\mathbb{R}^n\backslash B_x}G(x,\,\xi)|f(y)|{\rm d}y&\lesssim&\sum_{k=1}^{\infty}\int_{2^kB_x\backslash 2^{k-1}B_x}\Big|\frac{\Omega(x-y)}{|x-y|^{n+1}}-\frac{\Omega(\xi-y)}{|\xi-y|^{n+1}}\Big|\\
&&\qquad\times\big(k+|\nabla A(y)-\langle \nabla A\rangle_{B_x}|\big)|f(y)|{\rm d}y\\
&\lesssim& M_{L\log L}f(x).
\end{eqnarray*}
Therefore, for each $\xi\in Q$,
\begin{eqnarray}\label{equa:3.19}
\big|T_A(f\chi_{\mathbb{R}^n\backslash B_x})(x)-T_A(f\chi_{\mathbb{R}^n\backslash B_x})(\xi)\big|\lesssim M_{L\log L}f(x).
\end{eqnarray}
Combining the estimates (\ref{equa:3.17})-(\ref{equa:3.19}) leads to that
$${\rm ess}\sup_{\xi\in Q}|T_A(f\chi_{\mathbb{R}^n\backslash 3Q})(\xi)|\leq CM_{L\log L}f(x)+T^*_Af(x).$$

We turn our attention to $\mathcal{M}_{T_A^*}$. Again, it suffices to verify that
for bounded functions $f_1,\dots,f_N$ with compact supports, there exists a sparse family $\mathcal{S}$, such that for a.\,\,e. $x\in\mathbb{R}^n$,
\begin{eqnarray}\label{equa:3.21}\big\|\{T_A^*f_k(x)\}\big\|_{l^q}\lesssim \mathcal{A}_{\mathcal{S},\,L\log L}(\|\{f_k\}\|_{l^q})(x),\end{eqnarray}
which, by Theorem \ref{th2.1}, can be reduce to proving that
\begin{eqnarray}\label{equa:3.20}\mathcal{M}_{T_A^*}f(x)\le CM_{L\log L}f(x)+T^*_Af(x).\end{eqnarray}
Let $Q\subset \mathbb{R}^n$ be a cube and $x,\,\xi\in Q$. Write
\begin{eqnarray*}
|T_A^*(f\chi_{\mathbb{R}^n\backslash 3Q})(\xi)|&\leq &|T_A^*(f\chi_{\mathbb{R}^n\backslash B_x})(x)-T_A^*(f\chi_{\mathbb{R}^n\backslash B_x})(\xi)|\\
&&+|T_A^*(f\chi_{\mathbb{R}^n\backslash B_x})(x)|+|T_A^*(f\chi_{B_x\backslash 3Q})(\xi)|\\
&\lesssim&\sup_{\epsilon>0}\big|T_{A,\,\epsilon}(f\chi_{\mathbb{R}^n\backslash B_x})(x)-T_{A,\,\epsilon}(f\chi_{\mathbb{R}^n\backslash B_x})(\xi)\big|\\
&&+T_A^*f(x)+M_{L\log L}f(x).
\end{eqnarray*}
A straightforward computation leads to that for each $\epsilon>0$,
\begin{eqnarray*}
&&\big|T_{A,\,\epsilon}(f\chi_{\mathbb{R}^n\backslash B_x})(x)-T_{A,\,\epsilon}(f\chi_{\mathbb{R}^n\backslash B_x})(\xi)\big|\\
&&\lesssim\int_{|x-y|>\epsilon,\,|\xi-y|<\epsilon}\frac{|\Omega(x-y)|}{|x-y|^{n+1}}|A(x)-A(y)-\nabla A(y)(x-y)||f\chi_{\mathbb{R}^n\backslash B_x}(y)|{\rm d}y\\
&&+\int_{|x-y|\leq \epsilon,\,|\xi-y|>\epsilon}\frac{|\Omega(\xi-y)|}{|\xi-y|^{n+1}}|A(\xi)-A(y)-\nabla A(y)(\xi-y)||f\chi_{\mathbb{R}^n\backslash B_x}(y)|{\rm d}y\\
&&+\int_{\mathbb{R}^n\backslash B_x}\Big|\frac{\Omega(x-z)}{|x-z|^{n+1}}A(x)-A(y)-\nabla A(y)(x-y)\\
&&\quad-\frac{\Omega(x-z)}{|x-z|^{n+1}}A(x)-A(y)-\nabla A(y)(x-y)\Big||f(y)|{\rm d}y\\
&&\quad={\rm E}_1+{\rm E}_2+{\rm E}_3.
\end{eqnarray*}
As in the proof of (\ref{equa:3.19}), we know that
$${\rm E}_3\lesssim M_{L\log L}f(x).$$
On the other hand, as in (\ref{equa:3.18}), we deduce that
\begin{eqnarray*}{\rm E}_1&\lesssim&\int_{\epsilon<|x-y|\leq 2\epsilon}\frac{|\Omega(x-y)|}{|x-y|^{n+1}}|A(x)-A(y)-\nabla A(y)(x-y)||f(y)|{\rm d}y\\
&\lesssim &M_{L\log L}f(x),
\end{eqnarray*}
and
\begin{eqnarray*}{\rm E}_2&\lesssim&\int_{\epsilon<|\xi-y|\leq 2\epsilon}\frac{|\Omega(\xi-y)|}{|x-y|^{n+1}}|A(\xi)-A(y)-\nabla A(y)(\xi-y)||f(y)|{\rm d}y\\
&\lesssim &M_{L\log L}f(x).
\end{eqnarray*}
(\ref{equa:3.20}) now follows from the estimates for ${\rm E}_1$, ${\rm E}_2$ and ${\rm E}_3$. This completes the proof of Theorem \ref{t1.1}.
\qed
\begin{example}\label{e4.1}\rm  Let us consider the operator $$T_Af(x)=\int_{\mathbb{R}}\frac{A(x)-A(y)- A'(y)(x-y)}{(x-y)^2}f(y){\rm d}y.$$
For $A$ on $\mathbb{R}$ such that $A'\in {\rm BMO}(\mathbb{R})$, $T_A$ is bounded on $L^p(\mathbb{R},\,w)$ for $p\in (1,\,\infty)$ and $w\in A_p(\mathbb{R})$. Now let $\delta\in (0,\,1/2)$ and
$$f(x)=x^{-1+\delta}\chi_{(0,\,1)}(x),\,\,w(x)=|x|^{(p-1)(1-\delta)}.$$
It is well known that $[w]_{A_p}\approx \delta^{-p+1}$ (see \cite{bu, chpp}). Also,
$\|f\|_{L^p(\mathbb{R}^n,\,w)}^p=\delta^{-1}.$ Let $A(y)=y\log (|y|)$. We know that
$A'(y)=1+\log |y|\in {\rm BMO}(\mathbb{R})$.
A straightforward computation leads to that for $x\in (0,\,1)$,
\begin{eqnarray*}
T_Af(x)&=&\int_0^1\frac{x\log x-y\log y-(1+\log y)(x-y)}{(x-y)^2}y^{-1+\delta}{\rm d}y\\
&=&x\int^1_0\frac{\log x-\log y}{(x-y)^2}y^{-1+\delta}{\rm d}y-\int^1_0\frac{1}{x-y}y^{-1+\delta}{\rm d}y\\
&=&x^{-1+\delta}\int^{\frac{1}{x}}_0\frac{\log \frac{1}{t}}{(1-t)^2}t^{-1+\delta}dt-x^{-1+\delta}\int^{\frac{1}{x}}_0\frac{1}{1-t}t^{-1+\delta}{\rm d}t.
\end{eqnarray*}
Recall that for $t\in (0,\,1)\cup(1,\,\infty)$, $\log \frac{1}{t}\geq 1-t$. Therefore, for $x\in (0,\,1)$,
\begin{eqnarray*}
|T_Af(x)|&\geq &x^{-1+\delta}\int^{1}_0\frac{\log \frac{1}{t}-1+t}{(1-t)^2}t^{-1+\delta}{\rm d}t\\
&\geq&x^{-1+\delta}\int^{1}_0\big(\frac{\log \frac{1}{t}}{1-t}-1\big)t^{-1+\delta}{\rm d}t\\
&\geq&x^{-1+\delta}\int^{1}_0\big(\log \frac{1}{t}-1\big)t^{-1+\delta}{\rm d}t\\
&=&(\delta^{-2}-\delta^{-1})x^{-1+\delta}\geq \frac{1}{2}\delta^{-2}x^{-1+\delta}.
\end{eqnarray*}
Therefore,
$$\|T_Af\|_{L^p(\mathbb{R},\,w)}\geq \frac{1}{2}\delta^{-2}\|f\|_{L^p(\mathbb{R},\,w)}.$$
This shows that the conclusion in Theorem \ref{t1.1} is sharp when $p\in (1,\,2]$.
\end{example}

\section{Proof of Theorem \ref{t1.2}}
We begin with a lemma.
\begin{lemma}\label{le4.1}Let $\beta\in [0,\,\infty)$, $\mathcal{S}$ be a sparse family and $\mathcal{A}_{\mathcal{S},\,L(\log L)^{\beta}}$ be the associated sparse operator. Then for $p\in (1,\,\infty)$, $\epsilon\in (0,\,1]$ and weight $u$,
$$\|\mathcal{A}_{\mathcal{S},\,L(\log L)^{\beta}}f\|_{L^p(\mathbb{R}^n,\,u)}\lesssim p'^{1+\beta}p^2\big(\frac{1}{\epsilon}\big)^{\frac{1}{p'}}\|f\|_{L^p(\mathbb{R}^n,\,
M_{L(\log L)^{p-1+\epsilon}}u)}.$$
\end{lemma}
\begin{proof}Denote by $\mathcal{A}^*_{\mathcal{S},\,L(\log L)^{\beta}}$ the adjoint operator of $\mathcal{A}_{\mathcal{S},\,L(\log L)^{\beta}}$. Then for suitable functions $f$ and $g$, and any $s\in (1,\,\infty)$,
\begin{eqnarray*}\Big|\int_{\mathbb{R}^n}\mathcal{A}^*_{\mathcal{S},\,L(\log L)^{\beta}}f(x)h(x){\rm d}x\Big|&\le & \sum_{Q\in\mathcal{S}}|Q|\langle|f|\rangle_Q
\|h\|_{L(\log L)^{\beta},\,Q}\\
&\lesssim &\frac{1}{(s-1)^{\beta}}|Q|\langle|f|\rangle_Q\Big(\frac{1}{|Q|}\int_{Q}|g(y)|^sdy\Big)^{\frac{1}{s}}.
\end{eqnarray*}
Repeating the argument used in the proof of Theorem 1.7 in \cite{lpr}, we deduce that for $p\in (1,\,\infty)$, $\epsilon\in (0,\,1)$ and weight $u$,
$$\|\mathcal{A}^*_{\mathcal{S},\,L(\log L)^{\beta}}f\|_{L^{p'}(\mathbb{R}^n,\,(M_{L(\log L)^{p-1+\epsilon}}u)^{1-p'})}\lesssim_n p'^{1+\beta}p^2\big(\frac{1}{\epsilon}\big)^{\frac{1}{p'}}\|f\|_{
L^{p'}(\mathbb{R}^n,\,u^{1-p'})}.$$
This, via a duality argument, shows that
$$\|\mathcal{A}_{\mathcal{S},\,L(\log L)^{\beta}}f\|_{L^p(\mathbb{R}^n,\,u)}\lesssim_{n}p'^{1+\beta}p^2\big(\frac{1}{\epsilon}\big)^{\frac{1}{p'}} \|f\|_{L^p(\mathbb{R}^n,\,M_{L(\log L)^{p-1+\epsilon}}u)}.$$
This completes the proof of Lemma \ref{le4.1}.
\end{proof}
The following Theorem is an improvement of Lemma 4.1 in \cite{ler4}, and the proof here is of independent interest.
\begin{theorem} \label{th4.2}
Let $\mathcal{S}$ be a sparse family and $\beta\in [0,\,\infty)$, $\mathcal{A}_{\mathcal{S},L(\log L)^{\beta}}$ be the associated sparse operator. Let  $\epsilon\in (0,\,1]$ and $u$ be a weight. Then for each
$\lambda>0$,
\begin{eqnarray*}&&u(\{x\in\mathbb{R}^n:\, \mathcal{A}_{\mathcal{S},\,L(\log L)^{\beta}}f(x)>\lambda\})\\
&&\quad\lesssim \frac{1}{\epsilon^{1+\beta}}
\int_{\mathbb{R}^n}\frac{|f(x)|}{\lambda}\log ^{\beta}\Big({\rm e}+\frac{|f(x)|}{\lambda}\Big)M_{L(\log L)^{\epsilon}}u(x){\rm d}x.\end{eqnarray*}
\end{theorem}

\begin{proof}By the well known one-third trick, we may assume that $\mathcal{S}\subset \mathscr{D}$ for some dyadic grid $\mathscr{D}$.
Now let $M_{\mathscr{D},\,L(\log L)^{\beta}}$ be the maximal operator defined by
$$M_{\mathscr{D},\,L(\log L)^{\beta}}f(x)=\sup_{Q\ni x\atop{Q\in\mathscr{D}}}\|f\|_{L(\log L)^{\beta},\,Q}.$$
Decompose the set $\{x\in\mathbb{R}^n:\,M_{\mathscr{D},\,L(\log L)^{\beta}}f(x)>1\}$ as
$$\{x\in\mathbb{R}^n:\,M_{\mathscr{D},\,L(\log L)^{\beta}}f(x)>1\}=\cup_{j}Q_j,$$
with $Q_j$ the maximal cubes in $\mathscr{D}$ such that $\|f\|_{L(\log L)^{\beta},\,Q_j}>1$. Therefore,
$$1 <\frac{1}{|Q_j|}\int_{Q_j}|f(y)|\log^{\beta} \big({\rm e}+|f(y)|\big){\rm d}y\lesssim 2^n.$$Let
$$f_1(y) = f(y)\chi_{\mathbb{R}^n\backslash \cup_jQ_j}(y); f_2(y) =
\sum_j
f(y)\chi_{Q_j}(y);$$
and
$$f_3(y) =\sum_j\|f\|_{L(\log L)^{\beta},\,Q_j}\chi_{Q_j}(y).$$
It is obvious that $\|f_1\|_{L^{\infty}(\mathbb{R}^n)}\lesssim 1$. Thus, by Lemma \ref{le4.1},
\begin{eqnarray}\label{equa:4.1}
&&u(\{x\in\mathbb{R}^n:\,\mathcal{A}_{\mathcal{S},\,L(\log L)^{\beta}}f_1(x)>1/2\})\lesssim \|\mathcal{A}_{\mathcal{S},\,L(\log L)^{\beta}}f_1\|_{L^q(\mathbb{R}^n,u)}^q\\
&&\quad\lesssim q'^{q(1+\beta)}\big(\frac{1}{\epsilon}\big)^{\frac{q}{q'}}\int_{\mathbb{R}^n}|f_1(y)|^q
M_{L(\log L)^{q-1+\epsilon/2}}u(y){\rm d}y\nonumber\\
&&\quad\lesssim\frac{1}{\epsilon^{1+\beta}}\int_{\mathbb{R}^n}|f_1(y)|
M_{L(\log L)^{\epsilon}}u(y){\rm d}y,\nonumber
\end{eqnarray}
if we choose $q=1+\epsilon/2$.

Now let $E=\cup_j4nQ_j$, and $u_E(y)=u(y)\chi_{\mathbb{R}^n\backslash E}(y).$ We can verify that
\begin{eqnarray}\label{equa:4.2}u(E)\lesssim \sum_{j}\inf_{z\in Q_j}Mu(z)|Q_j|\lesssim \int_{\mathbb{R}^n}|f(y)|\log^{\beta}
\big({\rm e}+|f(y)|\big)Mu(y){\rm d}y
\end{eqnarray}
and for each $j$ and $\gamma\in [0,\,\infty)$
$$\sup_{y\in Q_j}M_{L(\log L)^{\gamma}}u_E(y)\approx\sup_{z\in Q_j}M_{L(\log L)^{\gamma}}u_E(z).$$
Note that $\|f_3\|_{L^{\infty}(\mathbb{R}^n)}\lesssim 1$ and
\begin{eqnarray*}
\|f_3\|_{L^1(\mathbb{R}^n,\,M_{L(\log L)^{\gamma}}u_E)}
&\lesssim &\sum_{j}\inf_{z\in Q_j}M_{L(\log L)^{\gamma}}u_E(z)|Q_j|\|f\|_{L(\log L)^{\beta},\,Q_j}\\
&\lesssim &\int_{\mathbb{R}^n}|f(y)|\log^{\beta} ({\rm e}+|f(y)|)M_{L(\log L)^{\gamma}}u_E(y){\rm d}y.
\end{eqnarray*}
If we can prove that for $x\in\mathbb{R}^n\backslash E$,
\begin{eqnarray}\label{equa:4.3}\mathcal{A}_{\mathcal{S},\,L(\log L)^{\beta}}f_2(x)\lesssim
\mathcal{A}_{\mathcal{S},\,L(\log L)^{\beta}}f_3(x), \end{eqnarray}
then by Lemma \ref{le4.1} and the inequality (\ref{equa:4.3}),
\begin{eqnarray}\label{equa:4.4}
&&u(\{x\in\mathbb{R}^n\backslash E:\,\mathcal{A}_{\mathcal{S},\,L(\log L)^{\beta}}f_2(x)>1\})\\
&&\quad\lesssim \|\mathcal{A}_{\mathcal{S},\,L(\log L)^{\beta}}f_3\|_{L^{q}(\mathbb{R}^n,u_E)}^{q}\nonumber\\
&&\quad\lesssim q'^{q(1+\beta)}\big(\frac{1}{\epsilon}\big)^{\frac{q}{q'}}\|f_3\|_{L^q(\mathbb{R}^n,\,M_{L(\log L)^{q-1+\epsilon/2}}u_E)}^q\nonumber\\
&&\quad\lesssim q'^{q(1+\beta)}\big(\frac{1}{\epsilon}\big)^{\frac{q}{q'}}\int_{\mathbb{R}^n}|f(y)|
\log^{\beta} ({\rm e}+|f(y)|)M_{L(\log L)^{q-1+\epsilon/2}}u_E(y){\rm d}y\nonumber\\
&&\quad\lesssim \frac{1}{\epsilon^{1+\beta}}\int_{\mathbb{R}^n}|f(y)|
\log^{\beta} ({\rm e}+|f(y)|)M_{L(\log L)^{\epsilon}}u(y){\rm d}y,\nonumber
\end{eqnarray}again we choose $q=1+\epsilon$. Our desired estimate for now follows from (\ref{equa:4.1}), (\ref{equa:4.2}) and (\ref{equa:4.4}) directly.

We now prove (\ref{equa:4.3}). For each fixed $x\in\mathbb{R}^n\backslash E$ and each cube $I\in \mathscr{D}$
containing $x$, note that $I\cap Q_j \not=\emptyset$ implies that $Q_j \subset I$. Thus, for each $\lambda > 0$, a
straightforward computation tells us that
\begin{eqnarray*}
&&\int_{I}\frac{|f_2(y)|}{\lambda}\log ^{\beta}\Big({\rm e}+\frac{|f_2(y)|}{\lambda}\Big){\rm d}y\\
&&\quad=\sum_{j:\, Q_j\subset I}\int_{Q_j}\frac{|f_2(y)|}{\lambda}\log ^{\beta}\Big({\rm e}+\frac{|f_2(y)|}{\lambda}\Big){\rm d}y\\
&&\lesssim \sum_{j:\, Q_j\subset I}\frac{\|f\|_{L(\log L)^{\beta},\,Q_j}}{\lambda}\log ^{\beta}\Big({\rm e}+\frac{\|f\|_{L(\log L)^{\beta},\,Q_j}}{\lambda}
\Big)\\
&&\qquad\times\int_{Q_j}\frac{|f(y)|}{\|f\|_{L(\log L)^{\beta},\,Q_j}}\log ^{\beta}\Big({\rm e}+\frac{|f(y)|}
{\|f\|_{L(\log L)^{\beta},\,Q_j}}\Big){\rm d}y\\
&&\quad\lesssim\sum_{j:\, Q_j\subset I}|Q_j|\frac{\|f\|_{L(\log L)^{\beta},\,Q_j}}{\lambda}\log ^{\beta}\Big({\rm e}+\frac{\|f\|_{L(\log L)^{\beta},\,Q_j}}{\lambda}
\Big),
\end{eqnarray*}
since for $t_1,\,t_2\in [0,\,\infty)$,
$$\log({\rm e}+t_1t_2)\lesssim \log ({\rm e}+ t_1)\log({\rm e}+ t_2).$$
On the other hand,
\begin{eqnarray*}&&\int_{I}\frac{|f_3(y)|}{\lambda}\log ^{\beta}\Big({\rm e}+\frac{|f_3(y)|}{\lambda}\Big){\rm d}y\\
&&\quad=\sum_{j:\, Q_j\subset I}\int_{Q_j}\frac{|f_3(y)|}{\lambda}\log ^{\beta}\Big({\rm e}+\frac{|f_3(y)|}{\lambda}\Big){\rm d}y\\
&&\quad=\sum_{j:\, Q_j\subset I}|Q_j|\frac{\|f\|_{L(\log L)^{\beta},\,Q_j}}{\lambda}\log ^{\beta}\Big({\rm e}+\frac{\|f\|_{L(\log L)^{\beta},\,Q_j}}{\lambda}
\Big).
\end{eqnarray*}
Therefore, for each $x\in\mathbb{R}^n\backslash E$ and $I\in \mathscr{D}$ containing $x$,
$$\|f_2\|_{L(\log L)^{\beta},\,I}\lesssim \|f_3\|_{L(\log L)^{\beta},\,I}.$$
The inequality (\ref{equa:4.3}) follows directly. This completes the proof of Theorem \ref{th4.2}.
\end{proof}

{\it Proof of Theorem \ref{t1.2}}. We only consider $T_A$. The argument for $T_A^*$ is similar. Applying the ideas used in \cite[p.608]{hp2} (see also the proof of Corollary 1.3 in \cite{ler4}, we deduce from Theorem \ref{th4.2} that for $w\in A_1(\mathbb{R}^n)$ and $\lambda>0$,
\begin{eqnarray*}&&w(\{x\in\mathbb{R}^n:\, \mathcal{A}_{\mathcal{S},\,L(\log L)^{\beta}}f(x)>\lambda\})\\
&&\quad\lesssim [w]_{A_1}\log^{\beta}({\rm e}+[w]_{A_{\infty}})
\int_{\mathbb{R}^n}\frac{|f(x)|}{\lambda}\log ^{\beta}\Big({\rm e}+\frac{|f(x)|}{\lambda}\Big)w(x){\rm d}x.\end{eqnarray*}
This, along with the inequality (\ref{equa:3.15'}) leads to our desired conclusion for $T_A$.
\qed
\begin{remark}Let $\epsilon\in (0,\,1]$ and $u$ a weight. By Lemma \ref{le4.1}, Theorem \ref{th4.2}, the estimates (\ref{equa:3.15'}) and (\ref{equa:3.21}), we know that for $p\in (1,\,\infty)$,
\begin{eqnarray*}
&&\|\{T_Af_k\}\|_{L^p(l^q;\,\mathbb{R}^n,\,u)}+\|\{T_A^*f_k\}\|_{L^p(l^q;\,\mathbb{R}^n,\,u)}\\
&&\quad\lesssim_{n}p'^{1+\beta}p^2\big(\frac{1}{\epsilon}\big)^{\frac{1}{p'}} \|f\|_{L^p(l^q;\,\mathbb{R}^n,\,M_{L(\log L)^{p-1+\epsilon}}u)}.\end{eqnarray*}
Moreover, for each fixed $\lambda>0$,
\begin{eqnarray*}
&&w\big(\big\{x\in\mathbb{R}^n:\,\|\{T_Af_k(x)\}\|_{l^q}+\|\{T_A^*f_k(x)\}\|_{l^q}>\lambda\big\}\big)\\
&&\quad\lesssim_{n}\frac{1}{\epsilon^{2}}
\int_{\mathbb{R}^n}\frac{\|\{f_k(x)\}\|_{l^q}}{\lambda}\log \Big({\rm e}+\frac{\|\{f_k(x)\}\|_{l^q}}{\lambda}\Big)M_{L(\log L)^{\epsilon}}u(x){\rm d}x.\end{eqnarray*}
These estimates extend and improve the main results in \cite{hy} and \cite{hl}.
\end{remark}

\end{document}